\documentclass[11pt,a4paper]{article}

\usepackage{amsmath,amsthm}
\usepackage{graphicx}

\theoremstyle{plain}
\newtheorem{theorem}{Theorem}
\newtheorem{lemma}[theorem]{Lemma}
\newtheorem{cor}[theorem]{Corollary}
\theoremstyle{definition}
\newtheorem{dnt}[theorem]{Definition}

\providecommand{\figwidth}{6cm}
\providecommand{\fullwidth}{12cm}

\title{Parity conditions for one-way rail networks}
\author{Dai Akita\thanks{The University of Tokyo, 7-3-1 Hongo Bunkyo-ku, Tokyo 113-8656, Japan. \texttt{d.akita@ne.t.u-tokyo.ac.jp}}
\and Daniel Schenz\thanks{Kobe University, 1-1 Rokkodai-cho, Nada-ku, Kobe, Japan.}}
\date{}

\begin{document}
\maketitle

\begin{abstract}
We present parity conditions under which a toy rail network is one-way,
i.e., whether a direction can be assigned across the network so that all train journeys are completely consistent with it
or completely consistent with its opposite.
We show that this problem is equivalent to determining the balance of a signed graph obtained from the network,
whose edges are assigned positive or negative signs.
Using signed-graph theory, we derive two equivalent parity conditions for one-wayness:
(i) every cycle must contain an even number of edges that join the same sides of switches,
and (ii) every cycle must contain an even number of angles at switches.
Signed-graph theory also offers an analytical criterion:
A connected network is one-way if and only if the smallest eigenvalue of its signed Laplacian matrix is zero,
suggesting a computational tool for evaluating one-wayness.
\end{abstract}

\section{Introduction}
\label{sec:intro}

Toy railroads (see Figure \ref{fig:toy}) are simple toys, 
yet they offer long hours of fun for children, adults, and mathematicians alike.
Indeed, several researchers have studied mathematical problems derived from toy trains.
Demaine et al. \cite{https://doi.org/10.48550/arxiv.1905.00518} considered the question
whether a train with finite length can move from one configuration to another
and analyzed the hardness of the problem using nondeterministic constraint logic,
which represents reconfiguration problems with constraint graphs.
Mitani \cite{mitani2018study} examined railway lines composed of circular arcs
and straight segments without branching.
He proved that any point in the plane can be reached by concatenating these elements,
and also studied how many pieces are required to form closed paths.

Garmo \cite{garmo1996random} considered toy railroads that consist of straight tracks,
curved tracks, switches, and bridges.
The existence of bridges opens up the possibility of non-planar networks.
Furthermore, he required that railroads be closed, i.e., they do not have dead ends.
In his model, a train moves forward along the rails and is never allowed to reverse;
in particular, it cannot travel directly between the two branch ends of a Y-shaped switch.
Such rail networks have several interesting characteristics (Figure~\ref{fig:example}).
In a \emph{functioning network} (Figure~\ref{fig:example}a) a train that starts at any point in any direction
can reach every other point, whereas in a \emph{malfunctioning network} (Figure~\ref{fig:example}b),
there are points that cannot be reached starting from certain other points going into a certain direction.
A \emph{one-way rail network} (Figure~\ref{fig:example}c) is one in which we can assign a fixed direction
(e.g., as shown by the arrows) to all pieces of track such that the direction of any possible train journey
through the network is either completely consistent with that direction or completely consistent with its opposite.
In a \emph{two-way rail network} (Figure~\ref{fig:example}d), on the other hand, trying to consistently assign
such a direction to all tracks will invariably lead to two opposing directions meeting at some point in the network,
so that a train journey across that point would not be completely consistent with either orientation.

\begin{figure}[tbp]
    \centering
    \includegraphics[width=\figwidth]{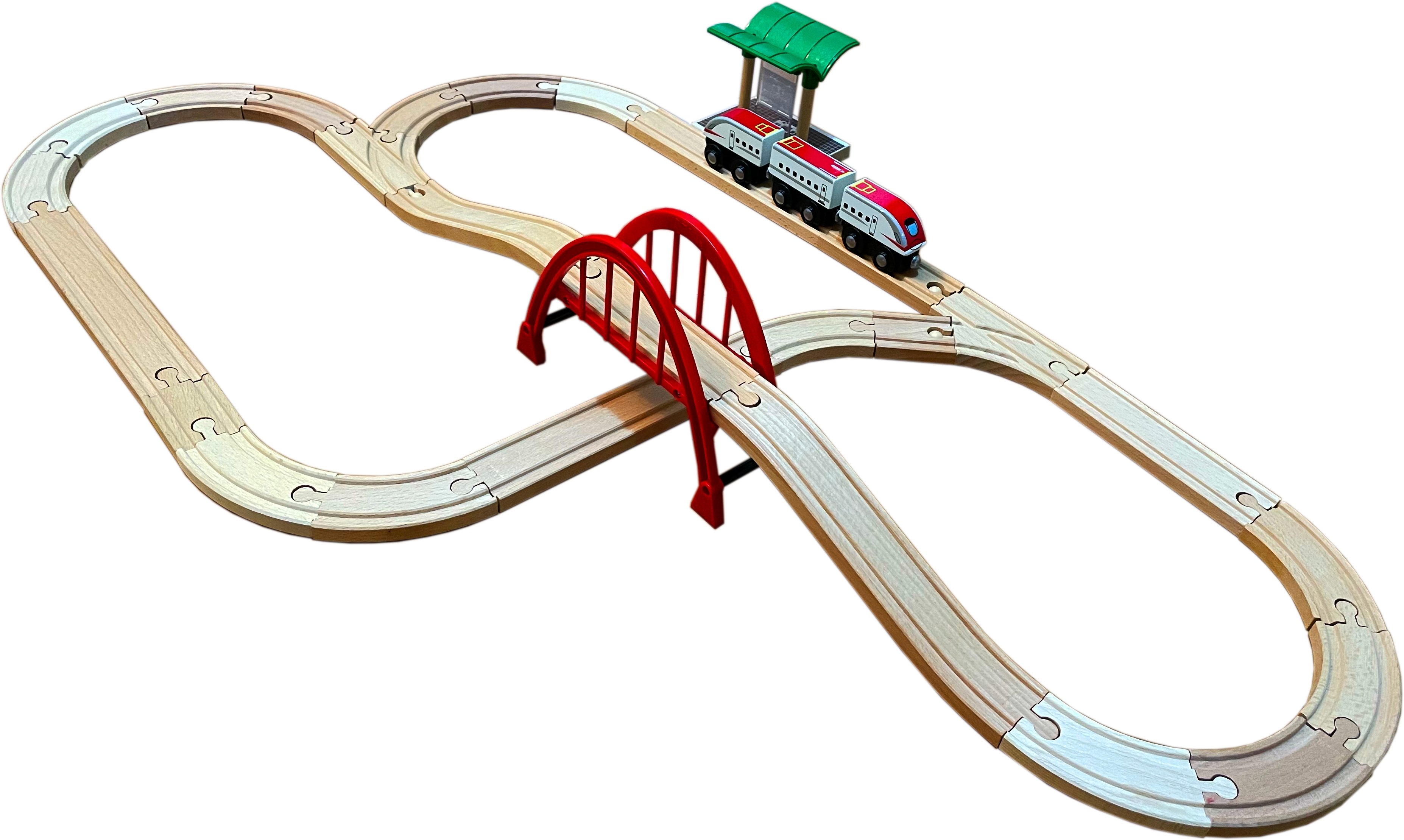}
    \caption{We consider toy railroads that consist of straight tracks, curved tracks, switches, and bridges.
    The existence of bridges opens up the possibility of non-planar networks.
    Furthermore, we require that railroads be closed, i.e., they do not have dead ends.}
    \label{fig:toy}
\end{figure}

\begin{figure}[tbp]
    \centering
    \includegraphics[width=9cm]{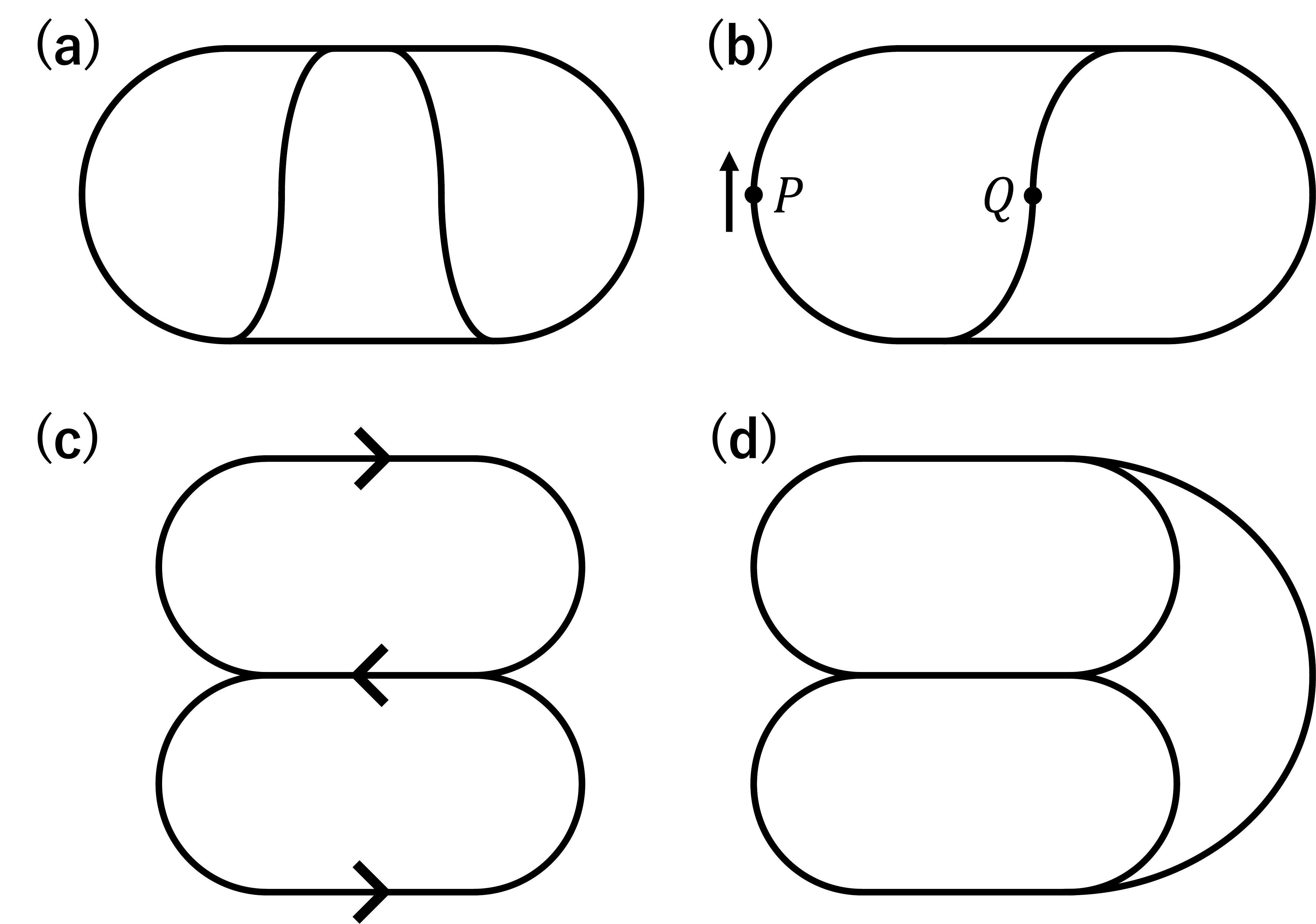}
    \caption{Characteristics of a rail network.
    (a) Functioning network: A train starting from any point and in any direction
    (although without allowing reversal of the direction of movement) can reach every other point.
    (b) Malfunctioning network: There exist a starting point and direction from which some points
    in the network cannot be reached.
    For example, if we start from $P$ clockwise, we cannot arrive at $Q$.
    (c) One-way network: We can assign a fixed direction (e.g., as shown by the arrows) to all pieces of track
    such that the direction of any possible train journey through the network is either completely consistent
    with that direction or completely consistent with its opposite.
    (d) Two-way network: Not a one-way network, i.e., it is not possible to assign a fixed direction
    that is consistent with all possible train journeys.}
    \label{fig:example}
\end{figure}

Garmo \cite{garmo1996random} formally treated these characteristics of rail networks
and considered the probability that a random rail network possesses them.
To compute these probabilities, he modeled rail networks as 3-regular multigraphs,
defined the properties mathematically,
and showed that the probabilities of being functioning and one-way converge asymptotically to 1/3 and 0, respectively,
as the network size tends to infinity.
Although he derived a necessary and sufficient geometric condition
for functioning — namely, the absence of an \emph{absorbing $(k,2l)$-configuration} — a graph-theoretic
condition for one-wayness has not been identified.

Here, we focus on the properties of a one-way rail network through the lens of signed-graph theory.
A \emph{signed graph}, introduced by Harary \cite{Harary1953}, is a graph
whose edges carry either a positive (+) or a negative (–) sign.
A signed graph is \emph{balanced} if every cycle contains an even number of negative edges.
The \emph{structure theorem} states that a signed graph is balanced
if and only if its vertex set can be partitioned into two disjoint subsets
such that positive edges join vertices within the same subset,
whereas negative edges join vertices across the two subsets.
Signed graphs have been widely applied to the analysis of social relationships
\cite{cartwright1956structural,10.1145/1753326.1753532,FacchettiIaconoAltafini2011,TraagDoreianMrvar2019}.

As we show in Theorem~\ref{th:onewaysigned}, a rail network is one-way precisely
when the signed graph derived from it is balanced.
This correspondence yields two equivalent parity conditions:
(i) every cycle must contain an even number of edges that join the same sides of switches,
and (ii) every cycle must contain an even number of angles at switches.
In Section \ref{sec:preliminaries} we formalize rail networks and one-wayness in graph-theoretic terms.
Section \ref{sec:signed} then establishes their equivalence to balance in the associated signed graph.

\section{Preliminaries}
\label{sec:preliminaries}

We formally define a rail network in a manner similar to that of Garmo~\cite{garmo1996random},
beginning with the notion of a switch — the basic building block of a rail network (Figure~\ref{fig:branch_graph}).

\begin{dnt}
A \emph{switch} is the graph $S(P, R)$ such that
\begin{equation}
    P = \{ s, b, b' \},
\end{equation}
\begin{equation}
    R = \{ r = sb, r' = sb' \}.
\end{equation}
We call the vertex $s$ the \emph{stem end}, and $b$ and $b'$ the \emph{branch ends}.
The edges $r$ and $r'$ are called \emph{branch rails}.
\end{dnt}

A rail network is a graph in which switches are connected to each other.

\begin{dnt}
Given a positive integer $N$ and $2N$ distinct switches
$S_i(P_i, R_i)$ for $i=1,\dots,2N$, let
$V = \bigcup_i P_i$ and $R = \bigcup_i R_i$.
Let $T = \{t_1,\ldots,t_{3N}\}$ be a set of unordered pairs of distinct vertices among $V$.
Then, with $E = R \cup T$, a graph $G(V, E)$ is called a \emph{rail network}
if $T$ is a perfect matching on $(V, T)$.
The edges in $T$ are called \emph{tracks}.
\end{dnt}

Note that a rail network is not necessarily simple:
For example, if a track $t$ connects a branch end and the stem end of the same switch,
$t$ is an additional edge between those vertices, parallel to the branch rail of the switch.

The distinctive property of a rail network is that not all transitions through a vertex are allowed:
At any given switch, a train may not move directly from one branch rail through $s$ to the other branch rail of that switch.
Consequently, a train can never traverse two branch rails in succession.
One can also confirm that a train cannot traverse two tracks successively. 
Motivated by these observations, we define the valid transitions of a train as a \emph{journey}.

\begin{figure}[tbp]
    \centering
    \includegraphics[width=\fullwidth]{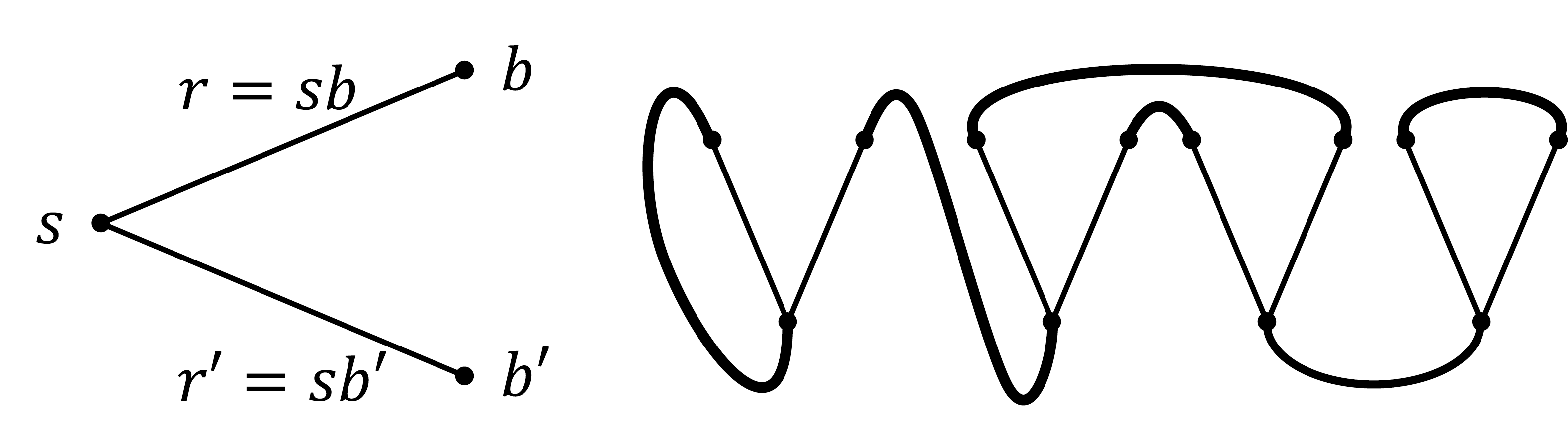}
    \caption{A switch (left) and a rail network (right). Thick edges in the rail network indicate tracks.}
    \label{fig:branch_graph}
\end{figure}

\begin{dnt}
A walk $j = v_0, e_1, v_1, \dots, e_n, v_n$ in a rail network $G(V,E)$ is called a \emph{journey}
if $j$ alternately traverses branch rails and tracks.
\end{dnt}

The main theme of this study is the \emph{one-way} rail network,
namely a network that admits an orientation
making every journey consistent with that orientation (or with its complete reversal).

\begin{dnt}
\label{def:orientation}
An orientation $\mathcal{O}$ on a rail network $G(V,E)$ is called a \emph{rail orientation}
if, in the digraph $D(V,\mathcal{O}(E))$, every branch end $b\in V$ has in-degree~1 and out-degree~1.
\end{dnt}

\begin{figure}[tbp]
    \centering
    \includegraphics[width=\fullwidth]{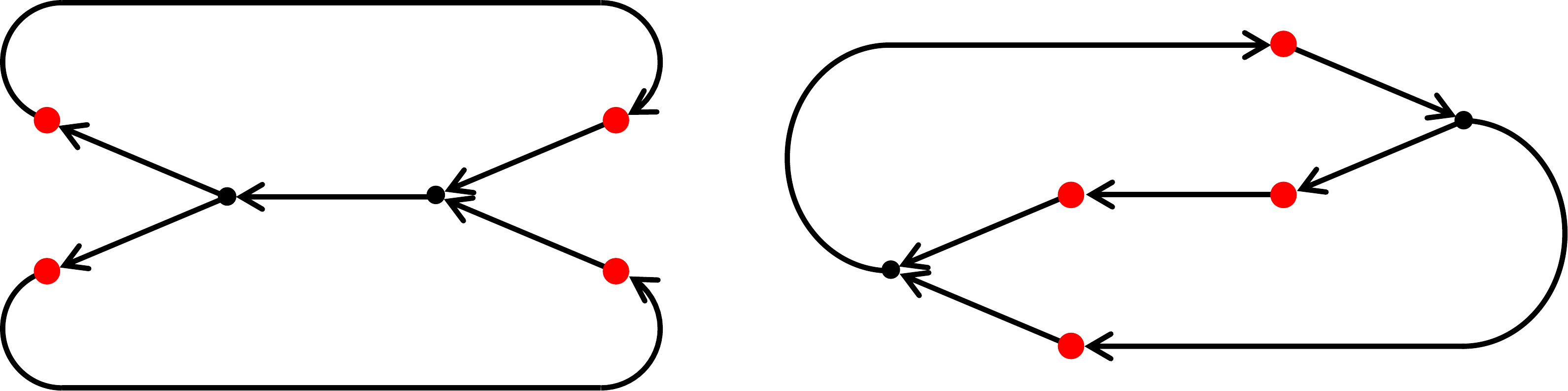}
    \caption{Rail orientations: Every branch end (large red dot) has in- and out-degree~1.}
    \label{fig:orientation}
\end{figure}

Let $\overline{\mathcal{O}}$ denote the reversal of~$\mathcal{O}$;
note that $\overline{\mathcal{O}}$ is also a rail orientation. 
For an arc $a$ in a digraph we write $\operatorname{head}(a)$ and $\operatorname{tail}(a)$ for its head and tail, respectively. 
We next define when a journey is consistent with a rail orientation.

\begin{dnt}
\label{def:consistent}
Given a journey $j = v_0 e_1 v_1 \cdots e_n v_n$ and a rail orientation $\mathcal{O}$, if
\begin{equation}
  \forall i \in \{1, \dots, n\}\quad \operatorname{head}\bigl(\mathcal{O}(e_i)\bigr)=v_i,
\end{equation}
then $j$ is said to be \emph{consistent} with $\mathcal{O}$.
\end{dnt}

We are now ready to give a formal definition of a one-way rail network.

\begin{dnt}
\label{def:oneway}
A rail network $G(V,E)$ is called \emph{one-way} if there exists a rail orientation $\mathcal{O}$
such that every journey $j$ is consistent with either $\mathcal{O}$ or its reversal $\overline{\mathcal{O}}$. 
Such an orientation $\mathcal{O}$ is referred to as a \emph{one-way orientation}. 
A rail network that is not one-way is called \emph{two-way}.
\end{dnt}

Thus, in a one-way rail network, every journey can be regarded as a directed walk
in the oriented digraph $D(V,\mathcal{O}(E))$ or in its reversal $\overline{D}(V,\overline{\mathcal{O}}(E))$. 
Because a journey traverses each track in only one direction,
a train cannot return to its starting point via the opposite direction.

We define a one-way rail network differently from Garmo~\cite{garmo1996random}.
His definition is inspired by the physical realization of wooden toy train tracks,
in which each rail piece has a \emph{plug} at one end and an \emph{antiplug} at the other,
so that a plug can be inserted into a neighbouring antiplug.
Garmo called a switch \emph{upward} when both branch ends have antiplugs and the stem end has a plug;
conversely, a switch is \emph{downward} when both branch ends have plugs and the stem end has an antiplug (Figure~\ref{fig:updown}).
Note that in a one-way rail network the numbers of upward and downward switches are equal,
because the total numbers of plugs and antiplugs must coincide.
Other switch types are possible — for example,
all three ends having the same fitting or the two branches having different fittings — but in Garmo’s framework
a rail network is one-way precisely when it can be built solely from upward and downward switches.

If we represent antiplugs by the heads of arcs and plugs by their tails,
such a one-way rail network can be described by an orientation.
Accordingly, we reinterpret upward and downward switches as follows.

\begin{figure}[tbp]
    \centering
    \includegraphics[width=8cm]{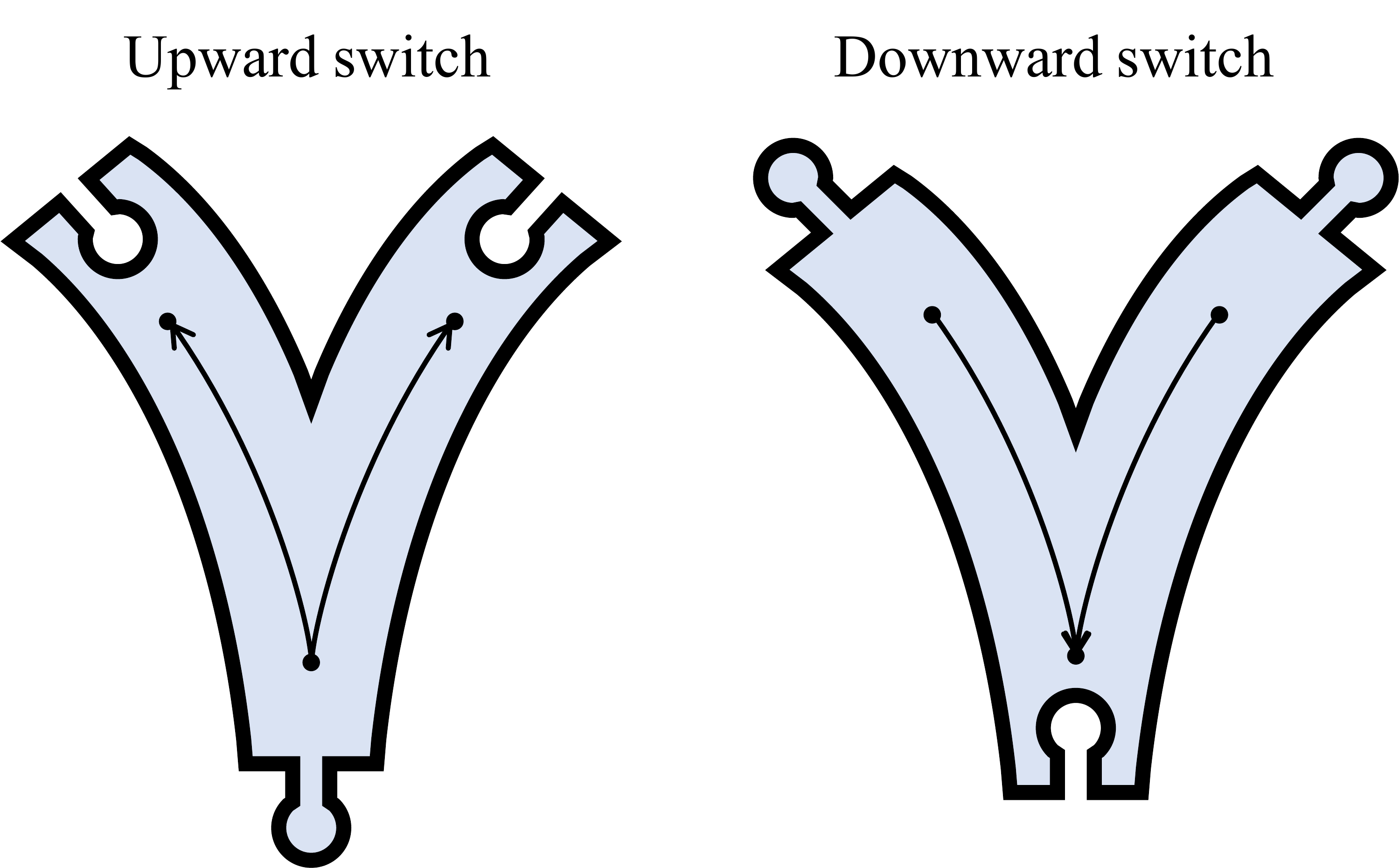}
    \caption{Upward and downward switches in the sense of Garmo~\cite{garmo1996random}. 
    In our terminology they correspond to upward and downward orientations.}
    \label{fig:updown}
\end{figure}

\begin{dnt}
\label{def:updown}
Let $G(V,E)$ be a rail network with rail orientation $\mathcal{O}$, and let
$S=(\{s,b,b'\},\{r,r'\})$
be a switch whose stem end $s$ is incident to a track $t_s$.
Then $S$ is said to be \emph{upward under $\mathcal{O}$} if
\begin{equation}
  \operatorname{head}\left( \mathcal{O}(t_s) \right)
  = \operatorname{tail}\left( \mathcal{O}(r) \right)
  = \operatorname{tail}\left( \mathcal{O}(r') \right) = s,
\end{equation}
and \emph{downward under $\mathcal{O}$} if
\begin{equation}
  \operatorname{tail}\left( \mathcal{O}(t_s) \right)
  = \operatorname{head}\left( \mathcal{O}(r) \right)
  = \operatorname{head}\left( \mathcal{O}(r') \right) = s.
\end{equation}
When the orientation is clear from context, we omit the phrase “under~$\mathcal{O}$.”
\end{dnt}

To compare our definition with Garmo’s, we first state a simple lemma.

\begin{lemma}
\label{lem:ss}
Let $\mathcal{O}$ be an arbitrary rail orientation.
If a journey $j$ does not visit any stem ends along the way,
$j$ is consistent with $\mathcal{O}$ or $\overline{\mathcal{O}}$.
\end{lemma}

\begin{proof}
Assume that $j$ runs between two stem ends $s$ and $s'$.
Because $j$ traverses branch rails and tracks alternately, only three patterns are possible (see Figure~\ref{fig:ss}):
\begin{enumerate}
\item $j = s t s'$,
\item $j = s r b t s'$,
\item $j = s r b t b' r' s'$,
\end{enumerate}
where $t$ is a track, $r$ and $r'$ are branch rails, and $b$ and $b'$ are branch ends.
Pattern 1 is trivially consistent with $\mathcal{O}$ or $\overline{\mathcal{O}}$.
For patterns 2 and 3, each branch end has in-degree 1 and out-degree 1 under any rail orientation, so the consistency holds.
If $j$ starts or ends at a branch end, it is a subsequence of pattern 3 and is therefore consistent as well.
\end{proof}

\begin{figure}[tbp]
    \centering
    \includegraphics[width=8cm]{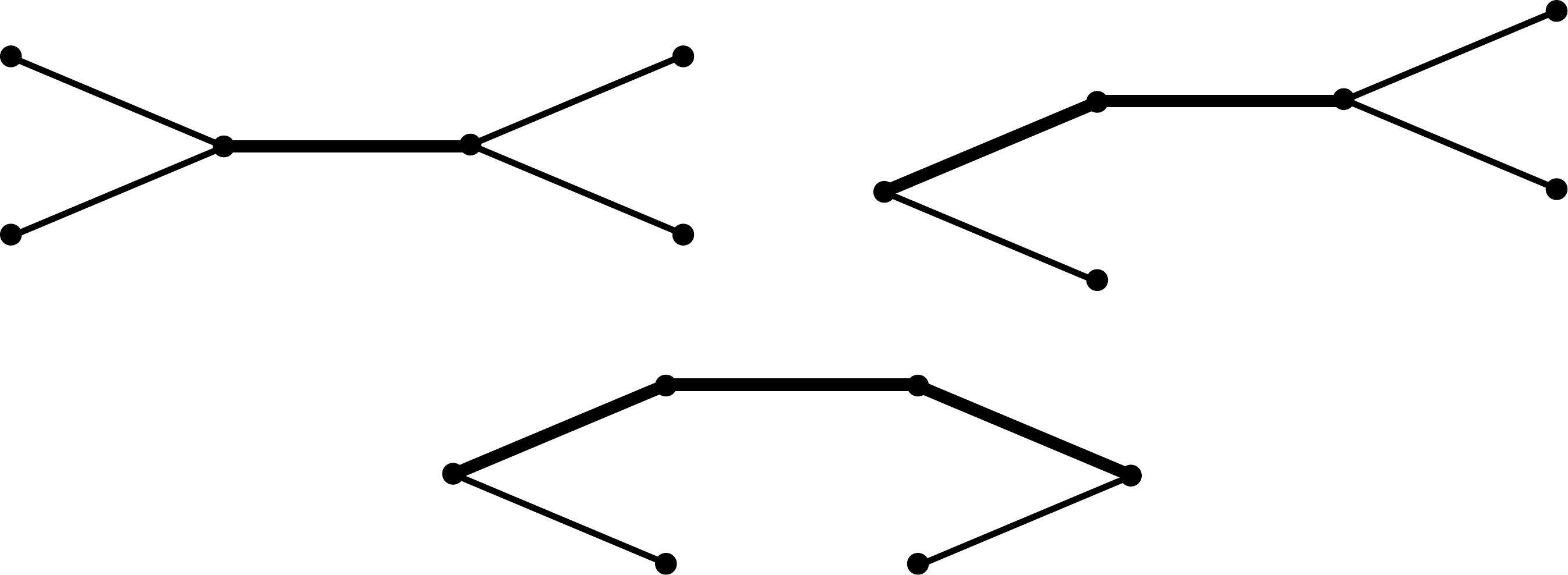}
    \caption{The three possible connection patterns between two stem ends that involve no other stem ends. 
    Upper left: The stem ends are connected directly by a track. 
    Upper right: They are connected via one branch end. 
    Bottom: They are connected via two branch ends.}
    \label{fig:ss}
\end{figure}

\begin{figure}[tbp]
    \centering
    \includegraphics[width=\figwidth]{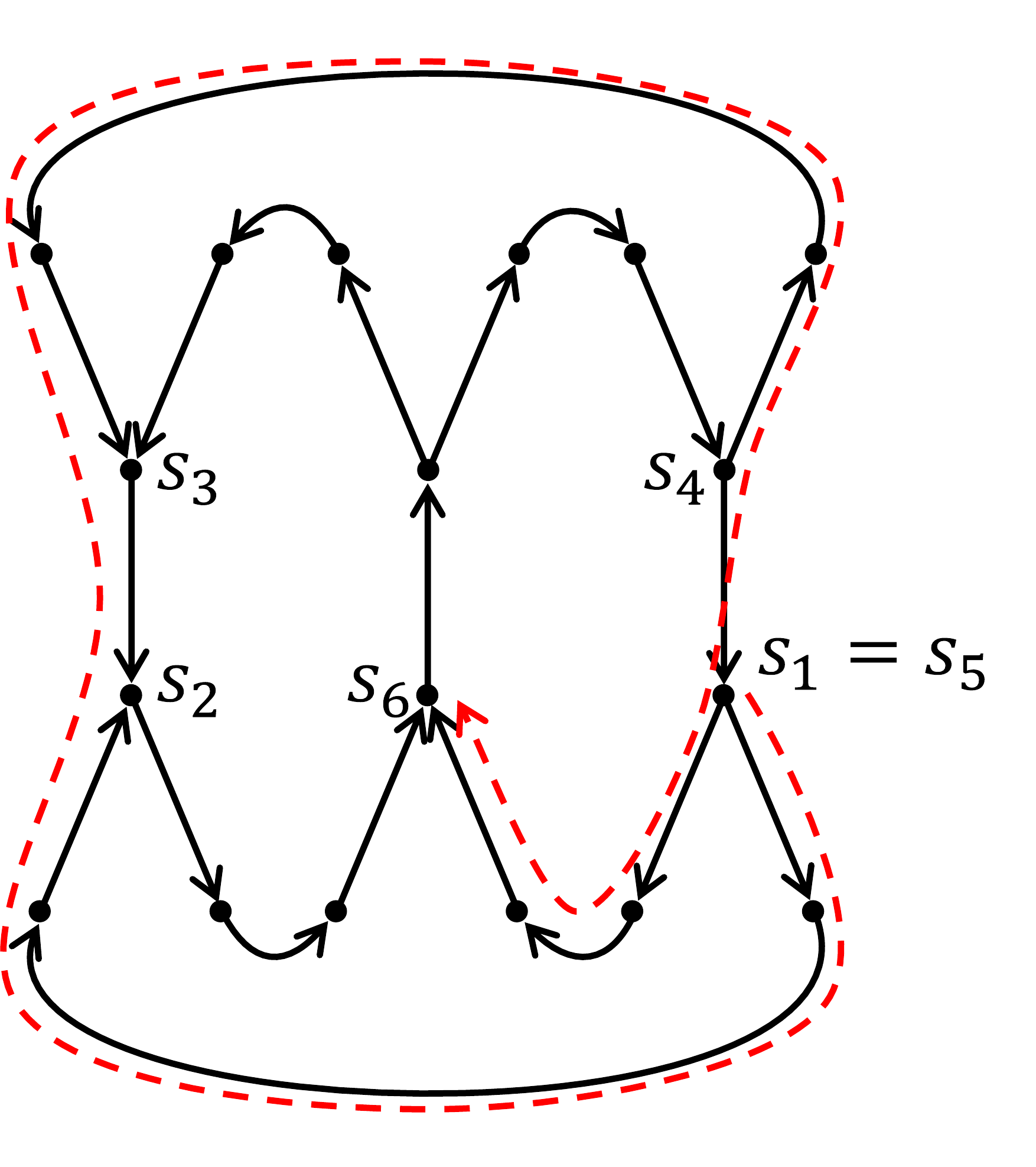}
    \caption{A journey (red dashed line) in a two-way rail network with an orientation. 
    The journey traverses stem ends $s_1, s_2, \ldots, s_6$.
    Note that each consistent segment starts and ends at a stem end rather than at a branch end.}
    \label{fig:journey_s}
\end{figure}

\begin{lemma}
\label{lem:equivalent}
A rail network $G(V,E)$ is one-way if and only if there exists a rail orientation $\mathcal{O}$
under which every switch is either upward or downward.
\end{lemma}

\begin{proof}
Assume that $G(V,E)$ is one-way and let $\mathcal{O}$ be its one-way orientation.
For an arbitrary switch $S=(\{s,b,b'\},\{r,r'\})$
let $t_s$ be the track incident to $s$ and consider the journeys $j_1=b r s t_s \dots$ and $j_2=b' r' s t_s \dots$.
By Definition \ref{def:updown} exactly one of
\begin{equation}
  \operatorname{head}\left( \mathcal{O}(r) \right)
  = \operatorname{tail}\left( \mathcal{O}(t_s) \right)
  = \operatorname{head}\left( \mathcal{O}(r') \right) = s
\end{equation}
or
\begin{equation}
  \operatorname{tail}\left( \mathcal{O}(r) \right)
  = \operatorname{head}\left( \mathcal{O}(t_s) \right)
  = \operatorname{tail}\left( \mathcal{O}(r') \right) = s
\end{equation}
holds, so $S$ is upward or downward.

Conversely, suppose there exists a rail orientation $\mathcal{O}$ under which every switch is upward or downward.
Let $j=v_0 e_1 v_1 \cdots e_n v_n$ be a journey and assume,
without loss of generality, that $\operatorname{head}\left( \mathcal{O}(e_1) \right)=v_1$.
Let $s_1,\dots,s_m$ be the (not necessarily distinct) stem ends encountered along $j$.
By Lemma \ref{lem:ss}, the segments of $j$ between successive stem ends are consistent with
$\mathcal{O}$ or $\overline{\mathcal{O}}$ (Figure~\ref{fig:journey_s} illustrates a case of two-way network).
Thus, it suffices to verify for each stem end $s_i$
that the two incident edges $\varepsilon_i$ and $\varepsilon'_i$ are oriented consistently with $\mathcal{O}$.
Let us assume $\varepsilon_i$ and $\varepsilon'_i$ are immediately before and after $s_i$, respectively, 
and $\operatorname{head}\left( \mathcal{O}(\varepsilon_i) \right)=s_i$.
Exactly one of $\varepsilon_i,\varepsilon'_i$ is a track and the other a branch rail.
If $\varepsilon_i$ is a track $t$ and $\varepsilon'_i$ a branch rail $r$,
then $s_i$ is upward and hence $\operatorname{tail}\left( \mathcal{O}(r) \right)=s_i$.
If $\varepsilon_i$ is a branch rail $r$ and $\varepsilon'_i$ a track $t$,
then $s_i$ is downward and thus $\operatorname{tail}\left( \mathcal{O}(t) \right)=s_i$.
Therefore every journey is consistent with $\mathcal{O}$, and the network is one-way.
\end{proof}

\section{Rail network as signed graph}
\label{sec:signed}

As we observed in the proof of Lemma~\ref{lem:equivalent},
the passing direction at each switch in a one-way rail network is determined by a train’s initial heading
and falls into one of two passing types, which we call \emph{upward} and \emph{downward}.  
Along a journey the passing direction is preserved
when the train traverses a track that connects a stem end and a branch end,
but it flips when the track connects two stem ends or two branch ends.  
This property suggests analyzing rail networks by means of the theory of signed graphs.

Signed graphs are graphs whose edges take the binary signs $+$ or $-$.  We adopt the following definition.

\begin{dnt}
A triplet $G(V,E,w)$ is called a \emph{signed graph}
if $(V,E)$ is a graph and $w:E\rightarrow\{+,-\}$ assigns each edge the sign $+$ (positive) or $-$ (negative).  
For $e\in E$ we say $e$ is \emph{positive} when $w(e)=+$ and \emph{negative} when $w(e)=-$.
\end{dnt}

Signed graphs are often used to model social relationships,
where negative edges represent adversarial ties~\cite{cartwright1956structural,10.1145/1753326.1753532,FacchettiIaconoAltafini2011,TraagDoreianMrvar2019}.  
In that context the dictum ``an enemy of an enemy is a friend'' leads to the notion of balance.

\begin{dnt}
A signed graph $G(V,E,w)$ is \emph{balanced} if every cycle in $G$ contains an even number of negative edges.
\end{dnt}

A fundamental result, sometimes called the \emph{structure theorem} of signed graphs, characterises balanced graphs as follows.

\begin{theorem}
\label{th:structure}
A signed graph $G(V,E,w)$ is balanced if and only if $V$ can be partitioned into two disjoint subsets $V_1$ and $V_2$
such that all edges within $V_1$ and $V_2$ are positive,
while every edge between $V_1$ and $V_2$ is negative.
\end{theorem}

See~\cite{Harary1953} for a proof.

Motivated by the change of passing direction discussed above,
we associate a signed graph with each rail network.

\begin{dnt}
\label{def:tracks}
Given a rail network $G(V,E)$, define $w:E\to\{+,-\}$ by
\begin{equation}
  w(e)=
  \begin{cases}
    -, & \text{if $e$ is a track connecting two stem ends or two branch ends},\\
    +, & \text{otherwise}.
  \end{cases}
\end{equation}
The resulting signed graph $G(V,E,w)$ is called the \emph{rail-signed graph} of the rail network.
\end{dnt}

Because the rail-signed graph is uniquely determined by the rail network,
we sometimes refer to it simply by the same symbol $G$.

We now show that the balance of the rail-signed graph exactly captures the one-way property of the rail network,
yielding a convenient parity criterion.
Figure~\ref{fig:vuvd} illustrates the concept of the proof:
a one-way rail network and its separation according to the structure theorem.

\begin{theorem}
\label{th:onewaysigned}
A rail network $G(V,E)$ is one-way if and only if its rail-signed graph $G(V,E,w)$ is balanced.
\end{theorem}

\begin{proof}
Assume first that $G$ is one-way, and let $\mathcal{O}$ be a one-way orientation.  
Assign every vertex the type \emph{upward} or \emph{downward} according to its switch.  
Let $V_u$ and $V_d$ denote the vertex sets of upward and downward switches, respectively.  

If $e$ joins two vertices of $V_u$ (or of $V_d$) then,
because one endpoint is a stem end and the other a branch end,
$e$ must be a track of the ``stem-branch'' kind or a branch rail; in either case $e$ is positive.  
On the other hand, any edge between $V_u$ and $V_d$ necessarily connects
two stem ends or two branch ends and is therefore negative.  
By Theorem~\ref{th:structure}, the rail-signed graph is balanced.

Conversely, suppose that $G(V,E,w)$ is balanced.  
Partition $V$ into $V_u$ and $V_d$ as in Theorem~\ref{th:structure}.  
Vertices of a switch belong to the same group since branch rails are positive.
Since the edges connecting vertices of the same group are positive,
if there is a track within $V_u$ (or $V_d$), it must connect a stem end and a branch end.
In addition, because the edges between $V_u$ and $V_d$ are negative,
they must be tracks connecting two stem ends or two branch ends.

Here, orient every branch rail inside $V_u$ from the stem end toward the branch end,
and every branch rail inside $V_d$ in the opposite direction.  
For a track within $V_u$ (resp.\ $V_d$) orient it from the branch end to the stem end
(resp.\ from the stem end to the branch end).  
Finally, orient each negative track from $V_u$ to $V_d$ if it connects two branch ends,
and from $V_d$ to $V_u$ if it connects two stem ends.  
The resulting orientation $\mathcal{O}$ is a rail orientation
which makes every switch in $V_u$ upward and every switch in $V_d$ downward,
hence $G$ is one-way.
\end{proof}

\begin{figure}[tbp]
  \centering
  \includegraphics[width=9cm]{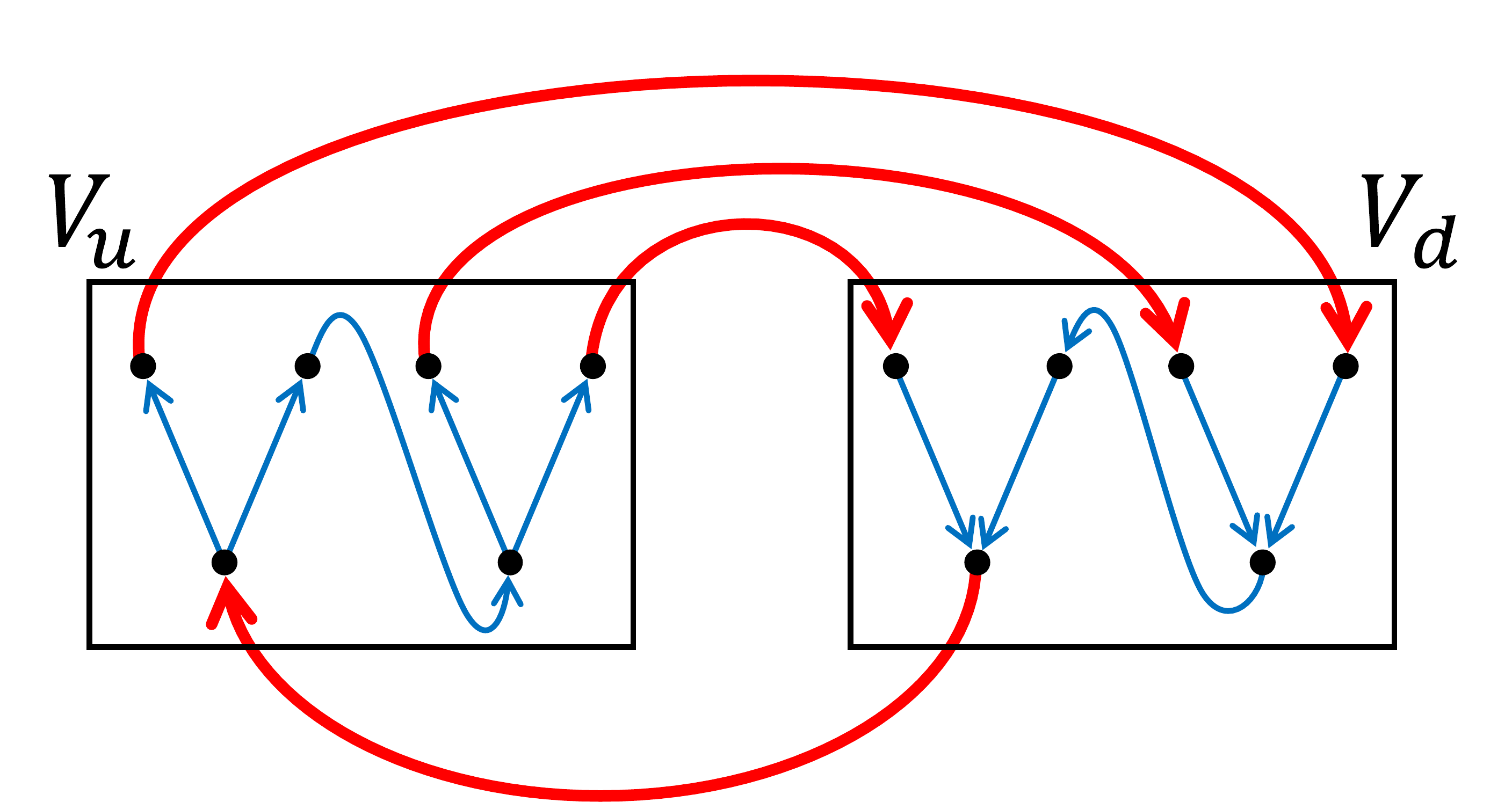}
  \caption{Interpreting a one-way rail network as a balanced signed graph.  
  Positive edges (thin blue) lie within $V_u$ or $V_d$,
  while negative edges (thick red) run between $V_u$ and $V_d$.}
  \label{fig:vuvd}
\end{figure}

\begin{cor}
\label{co:counting}
A rail network is one-way if and only if every cycle contains an even number of negative edges.
\end{cor}

\begin{proof}
Immediate from Theorem~\ref{th:onewaysigned} and the definition of a balanced signed graph.
\end{proof}

Thus, we can decide whether a rail network is one-way by counting the negative tracks in each cycle.
To obtain a visually simpler counting method,
we focus on the special role of \emph{angles} in a rail network.

\begin{dnt}
A sequence $r s r'$ consisting of the stem end $s$ and the branch rails $r,r'$
of the same switch is called an \emph{angle}.
\end{dnt}

An angle is therefore an acute turning point on a path of the rail network
and cannot be part of a journey.
We next show that one-wayness can be detected by counting angles as well.

\begin{theorem}
\label{the:oddeven}
Given a cycle of a rail network $G(V,E)$,
the number of negative tracks in the cycle is odd
if and only if the number of angles in the cycle is odd.  
Consequently, $G$ is one-way if and only if no cycle of $G$ contains an odd number of angles.
\end{theorem}

\begin{proof}
Let the cycle be $v_0 e_0 v_1 e_1 \dots v_n e_n v_0$,
and consider the closed sequence $w=p_0\dots p_n p_0$ where
\begin{equation}
  p_i=
  \begin{cases}
    sr, & \text{if $v_i$ is a stem end and $e_i$ is a branch rail},\\
    st, & \text{if $v_i$ is a stem end and $e_i$ is a track},\\
    br, & \text{if $v_i$ is a branch end and $e_i$ is a branch rail},\\
    bt, & \text{if $v_i$ is a branch end and $e_i$ is a track}.
  \end{cases}
\end{equation}
The sequence $w$ is a closed directed walk in the state diagram of Figure~\ref{fig:statesequence}
which shows possible transitions for the sequence $p_i$.
Thick blue arcs $(st,sr)$ and $(bt,br)$ represent transitions through a negative track;
solid black arcs $(sr,bt)$, $(bt,sr)$, $(st,br)$, $(br,st)$ connect switches with positive tracks;
thick red dashed arc $(br,sr)$ corresponds to angle and is not traversable by a train.
Dividing the nodes into the two groups $\{sr,bt\}$ and $\{st,br\}$,
we see that every closed walk crosses negative tracks and angles in total an even number of times,
so the numbers of negative tracks and angles share the same parity.
\end{proof}

\begin{figure}[tbp]
  \centering
  \includegraphics[width=\dimexpr\figwidth*8/10]{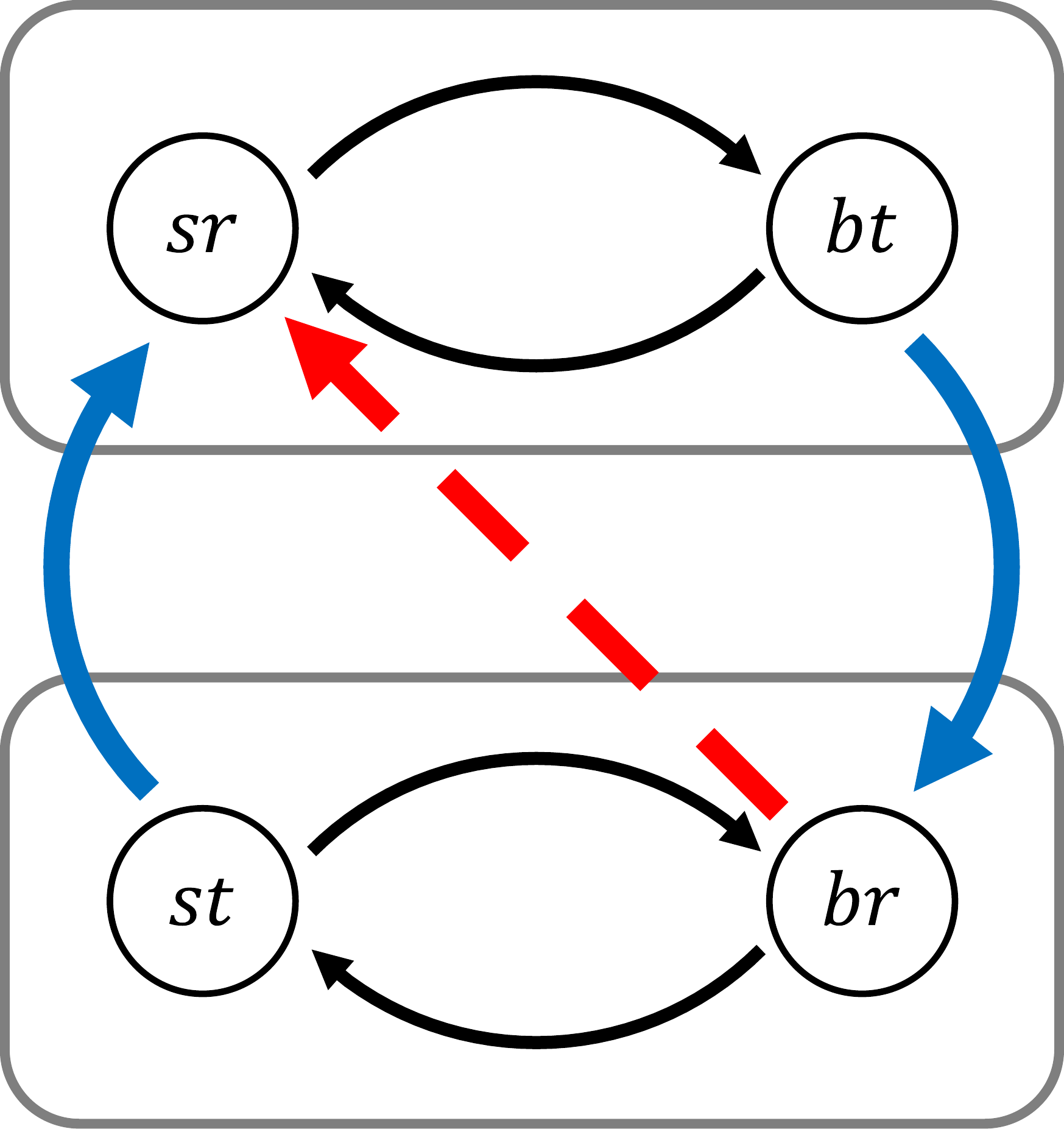}
  \caption{Possible state transitions within a cycle of a rail network.  
  Thick blue arcs: moves along negative tracks;
  solid black arcs: moves along positive tracks;  
  dashed red arc: passage through an angle.}
  \label{fig:statesequence}
\end{figure}

Figure~\ref{fig:rule} illustrates these rules.
In Figure~\ref{fig:rule}a the network is one-way:
All tracks are negative and every cycle contains an even number of negative tracks and angles.
Figure~\ref{fig:rule}b shows a two-way network shaped like the yin-yang symbol.
Two outer tracks are positive and the central track is negative,
so each of the two main cycles contains one negative track and one angle.
Figure~\ref{fig:rule}c modifies Figure~\ref{fig:rule}a by adding a track on the right,
creating a new cycle with three negative tracks and three angles.
The network is therefore no longer one-way.

\begin{figure}[tbp]
  \centering
  \includegraphics[width=\dimexpr\figwidth*2]{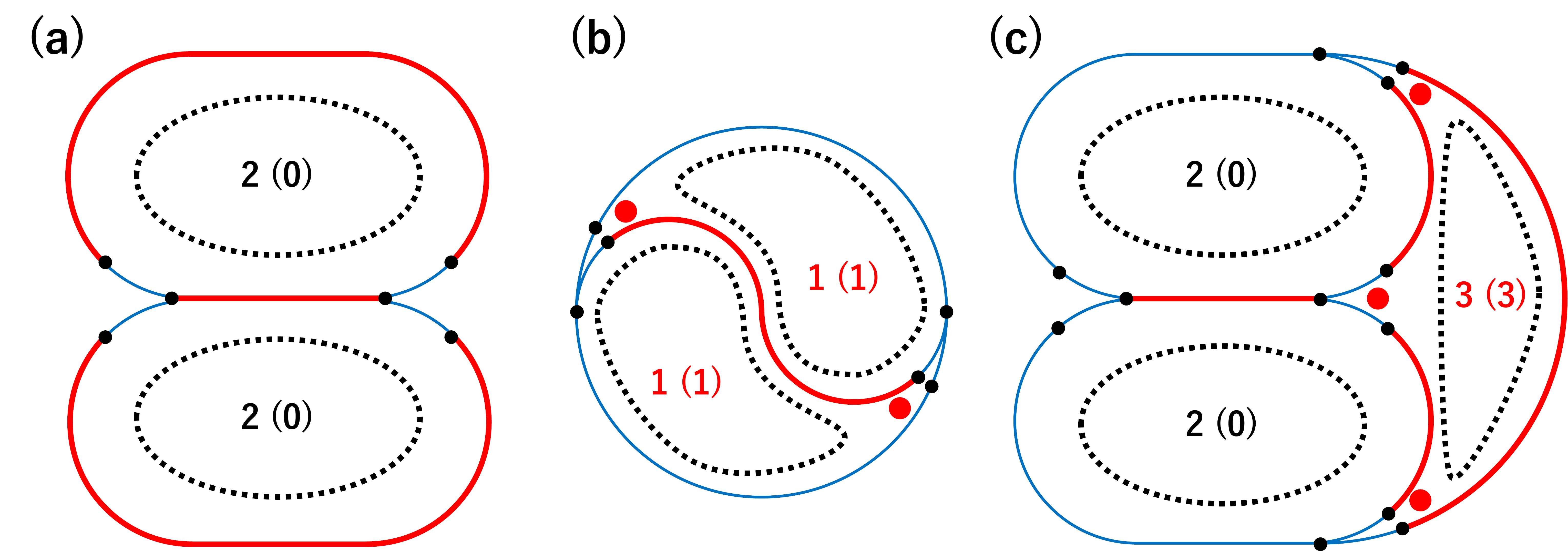}
  \caption{Demonstration of the counting rules.  
  Positive and negative tracks are drawn as thin blue and thick red lines, respectively;
  angles are marked by red dots.  
  Dotted curves highlight the cycles considered.  
  Numbers indicate the count of negative tracks in each cycle and,
  in parentheses, the count of angles.  
  (a) A one-way network with an even count in every cycle.  
  (b), (c) Two-way networks each containing a cycle with an odd count (red numbers).}
  \label{fig:rule}
\end{figure}

Representation of a rail network as a signed graph enables us to apply the theory of signed graphs, whose key theorem provides a necessary and sufficient condition for balance in terms of the Laplacian of a signed graph.

\begin{dnt}
\label{def:laplacian}
Given a signed graph $G(V,E,w)$ with $V=\{v_1,\dots,v_n\}$,
let $a^+_{ij}$ and $a^-_{ij}$ $(i\neq j)$ denote the numbers of positive and negative edges, respectively, between $v_i$ and $v_j$,
and let $a^+_{ii}$ and $a^-_{ii}$ be the numbers of positive and negative loops at $v_i$, respectively.
The \emph{Laplacian} $L=(L_{ij})$ is defined entrywise by
\begin{equation}
  L_{ij}=-(a^+_{ij}-a^-_{ij}),
\end{equation}
for $i \neq j$, and for the diagonal entries, 
\begin{equation}
  L_{ii}=4\,a^-_{ii}+\sum_{j\neq i}\left(a^+_{ij}+a^-_{ij}\right).
\end{equation}
\end{dnt}

The Laplacian matrix is positive semidefinite \cite{zaslavsky2013signedgraphsgeometry},
and its smallest eigenvalue decides the balance if the signed graph is connected.

\begin{theorem}
\label{th:laplacian}
A connected signed graph is balanced if and only if the smallest eigenvalue of its Laplacian is zero.
\end{theorem}

While proofs of this theorem for simple signed graphs can be found in the literature on signed graphs, e.g., \cite{TraagDoreianMrvar2019,Hou01012003},
there are fewer treatments of more general, signed multigraphs.
Zaslavsky's work \cite{zaslavsky2013signedgraphsgeometry}, however, considers
general signed graphs with loops and parallel edges,
and yields the following result\footnote{
More precisely, \cite{zaslavsky2013signedgraphsgeometry} also considered half edges and loose edges,
and Theorem \ref{th:rank} follows directly from Theorems 3.6 and 3.7 of \cite{zaslavsky2013signedgraphsgeometry}.
} from which Theorem \ref{th:laplacian} follows.

\begin{theorem}
\label{th:rank}
Let $G(V, E, w)$ be a signed graph with $n$ vertices, and let $k$ denote the number of balanced connected components.
Then,
\begin{equation}
  \operatorname{rank} L = n - k.
\end{equation}
\end{theorem}

\begin{proof}[Proof of Theorem \ref{th:laplacian}]
If $G(V, E, w)$ is balanced, then $k=1$ and $L$ is not full rank.
Therefore, there exists a zero eigenvalue of $L$.
Conversely, if $L$ has a zero eigenvalue, it follows that $k>0$ by Theorem \ref{th:rank};
since $G$ is connected, necessarily $k=1$, and hence $G(V, E, w)$ itself is balanced.
\end{proof}

We now obtain the following analytic criterion directly from Theorem \ref{th:laplacian}.

\begin{cor}
\label{co:analytic}
A connected rail network is one-way if and only if
the smallest eigenvalue of the Laplacian of the associated rail-signed graph is zero.
\end{cor}

The Laplacian is not only a tool for detecting one-way networks;
it also yields a measure of deviation from the one-way property.
Let $\mathbf{x}$ be a $\{-1, 1\}$-valued vector whose length equals the number of vertices.
The quadratic form $\mathbf{x}^\top L \mathbf{x}$ expands as
\begin{align}
  & \mathbf{x}^\top L \mathbf{x} = \sum_i \sum_j x_i x_j L_{ij} \notag \\
  & = \sum_i x_i^2 \left( 4 a^-_{ii} + \sum_{j \neq i} \left( a^+_{ij} + a^-_{ij} \right) \right)
  - \sum_i \sum_{j \neq i} x_i x_j \left( a^+_{ij} - a^-_{ij} \right) \notag \\
  & = 4 \sum_i a^-_{ii} + \sum_i \sum_{j > i} \left( a^+_{ij} \left( x_i - x_j \right)^2 + a^-_{ij} \left( x_i + x_j \right)^2 \right) \notag \\
  & = 4 \left( \sum_i a^-_{ii} + \sum_i \sum_{j > i} \left( a^+_{ij} \chi_{x_i \neq x_j} + a^-_{ij} \chi_{x_i = x_j} \right) \right),
\end{align}
where $\chi_P$ equals 1 if $P$ holds and 0 otherwise.
The formula means that the quadratic form is equal to 
four times the quantity known as the \emph{line index of imbalance}
(also called the \emph{frustration index})\cite{TraagDoreianMrvar2019,HararyNormanCartwright1965}\footnote{
See also \cite{TraagDoreianMrvar2019}.
However, the step from (8.29) to (8.30) in the paper is off by a factor of two:
only one of $(i,j)$ and $(j,i)$ is counted and (8.30) states that
$x^\top L x$ equals twice the line index of imbalance rather than four times.
},
which is the number of edges violating the partition rule of balance
in the partitioning $V_1 = \{v_i | x_i = 1\}$ and $V_2 = \{v_i | x_i = -1\}$.
Therefore, the minimum value of $\mathbf{x}^\top L \mathbf{x}$ is zero for a balanced signed graph
and larger for a signed graph that deviate more from balance,
suggesting a quantitative measure of the bidirectionality of a rail network.

\section{Concluding remarks}

We have proposed a graph-theoretic framework for analysing directionality in toy rail networks.
By assigning a sign to each track based on the type of connection,
we introduced the \emph{rail-signed graph} of a network
and showed that its balance is equivalent to the one-way property.
This equivalence provides novel methods to determine one-wayness, namely,
parity tests based on the numbers of negative tracks or angles in every cycle.
Furthermore, the smallest eigenvalue of the signed Laplacian gives an analytical way
to assess one-wayness and offers a quantitative measure of how far a given network deviates from it.
Applying this theory to real transport systems may offer new insights into the design of unidirectional transport networks.

\section*{Acknowledgements}
We would like to thank the Hatakeyama children for inviting us to play with their train sets, inspiring us to ask these questions, and Dr. Wataru Komatsubara for providing us with Plarails and fun while working out an answer.


\end{document}